\documentclass[opre,nonblindrev]{informs3} 

\DoubleSpacedXI 

\usepackage{endnotes}
\let\footnote=\endnote

\usepackage{natbib}
 \bibpunct[, ]{(}{)}{,}{a}{}{,}%
 \def\bibfont{\small}%

\usepackage{subfigure}
\usepackage{lscape,booktabs,longtable}
\usepackage{multirow}
\usepackage{xcolor}
\usepackage{amsfonts}
\usepackage{latexsym}
\usepackage{graphicx}
\usepackage{amssymb}
\usepackage{algorithm}
\newcommand{\secref}[1]{\S\ref{#1}}

\newcommand{\mytag}[2]{%
  \text{#1}%
  \@bsphack
  \protected@write\@auxout{}%
         {\string\newlabel{#2}{{#1}{\thepage}}}%
  \@esphack
}

\TheoremsNumberedThrough     
\ECRepeatTheorems

\EquationsNumberedThrough    

\MANUSCRIPTNO{}

\begin{document}

\RUNTITLE{The Continuous Joint Replenishment Problem is Strongly $\mathcal{NP}$-Hard}
\TITLE{The Continuous Joint Replenishment Problem is Strongly $\mathcal{NP}$-Hard}
\ARTICLEAUTHORS{%
\AUTHOR{ Alexander Tuisov, Liron Yedidsion  \thanks{Corresponding author. Email -
lirony@ie.technion.ac.il}\\}}

\ABSTRACT{
The Continuous Periodic Joint Replenishment Problem (CPJRP) has been one of the core and most studied problems in supply chain management for the last half a century. Nonetheless, despite the vast effort put into studying the problem, its complexity has eluded researchers for years. Although the CPJRP has one of the tighter constant approximation ratio of 1.02, a polynomial optimal solution to it was never found. Recently, the discrete version of this problem was finally proved to be $\mathcal{NP}$-hard. In this paper, we extend this result and finaly prove that the CPJRP problem is also strongly $\mathcal{NP}$-hard. 
}
\maketitle


\section{Introduction}

The joint replenishment problem (JRP) is a basic problem in the field of inventory management. The JRP aims to synchronize orders of different commodities so as to order them together and save costs. In this research, we refer to the schedule of the replenishment times for each commodity as the ordering policy. Whenever a commodity is ordered, it incurs a fixed
ordering cost as well as linear holding costs
that are proportional to the amount of the commodity held in storage. Linking
all commodities, a joint ordering cost is incurred whenever one or more
commodities are ordered. The objective of JRP is to minimize the sum of
ordering and holding costs. \\There are many
distinctions between the different JRP models studied in the literature, which
we elaborate upon in Section \ref{LR}. In this research, we study the continuous periodic
review JRP (CPJRP) model with continuous infinite horizon, and steady (stationary) demand. That is, facing a constant demand, we need to
minimize the average periodic holding and setup costs of all commodities as
well as the joint ordering costs. However, in the CPJRP the orders of each
commodity are placed periodically. The cycle times for each
commodity's orders are pre-determined and inflexible. The joint replenishment
is continuously reviewed and is paid only at times where at least one
commodity is being ordered.
Of all the different variations of the JRP model, the CPJRP is by far the most studied one.\\
Except for the non-stationary demand JRP model, the complexity of all JRP variations has been an open question for more than five decades, until recently when \citet{CY2018} finally resolved the complexity of the discrete version of our problem, namely the discrete periodic JRP (DPJRP).
In this research we extend the result of \citet{CY2018}, and prove that the CPJRP is strongly $\mathcal{NP}$-hard.

\subsection{\label{LR}Literature review}

Many variations of the JRP problem have been studied in the literature. \citet{ST2011} distinguished between some characteristics of this
problem.
\begin{itemize}
\item Commodity order policy constraints: There are three types of order
policy constraints for the JRP. The first model requires a periodic ordering
policy, also called a periodic review. A periodic review is the policy where for
each commodity we must pre-determine a cycle time. An order will occur at each
multiple of that cycle time. We refer to this model as the PJRP. The
second model does not require a cycle time for each commodity. However, it
requires a cyclic ordering policy. We refer to this model as the cyclic JRP. This model has no limits on the ordering policy.

\item Joint order policy constraints: The joint ordering cost in the JRP
model is a complicated function of the inter-replenishment times, so it is
often assumed that joint orders are placed periodically, even if some joint
orders are empty and the cycle times of the commodities are always a
multiple of the joint order cycle time. This model is often referred to
as Strict JRP. We refer to the model with continuous joint replenishment
review, where a joint order is placed only at the periods where at least one
commodity is ordered, as General JRP, or simply as JRP.

\item Demand type: We make a distinction between problems with
stationary demand for each commodity and problems with fluctuating demand.

\item Time horizon: The time horizon defines the horizon for which one must
plan an order policy. We distinguish between the problem with infinite horizon
and the problem with finite horizon.

\item Solution integrality: The integrality of the solution determines whether
the ordering policy will be integral or not.
\end{itemize}
Our focus in this research is the periodic, general, continuous time JRP with stationary demand and infinite horizon, referred to as the CPJRP.
The JRP is a special case of the One-Warehouse-N-Retailers problem (OWNR), which deals
with a single warehouse receiving goods from an external supplier and
distributing to multiple retailers. The warehouse could also serve as a
storage point. JRP in particular is a special case of the OWNR with a very
high warehouse holding cost. \cite{ARJ1989} stated that since JRP
is a special case of OWNR, proving JRP hardness also proves the hardness of OWNR.

\textbf{Strict PJRP. }The problem of Strict PJRP was well covered in the
reviews by \cite{GS1989} and \cite{KG2008}.
Many research attempts have been made in order to find efficient solutions to
the Strict PJRP since the early 1970s. Heuristic approaches were suggested in \cite{S1971}, \cite{N1973}, \cite{S1976}, \cite{V1993} \cite{KR1983, GB1979, KR1991, GD1993, V1996, FM2001, v2002, PD2004, WFD1997}, and \cite{O2005}.

Since JRP is a special case of OWNR, results regarding the OWNR hold for JRP
as well. Hence, the following results are applicable for JRP. A prominent
advancement in the study of OWNR, the optimal Power-of-Two policy, was
achieved by \cite{R1985}. This policy could be computed in $O(n\log n)$
time. \cite{R1985} proved that the cost of the best power-of-two policy
achieves $98\%$ of an optimal policy ($94\%$ if the base planning period is
fixed). In other words, he suggested a 1.02-approximation (1.064 for the
fixed-based planning period) for JRP, where a $\rho$-approximation algorithm
is an algorithm that is polynomial with respect to the number of elements, and
the ratio between the worst case scenario solution and the optimal solution is
bounded by a constant, $\rho$. Based on \cite{R1985}, \cite{J1985} proposed
an efficient algorithm that offers a replenishment policy in which the cost is
within a factor of $\sqrt{\frac{9}{8}}\approx1.06$ of the optimal solution.
This approximation was later improved to $\frac{1}{\sqrt{2}\ln2}$ for a
non-fixed-based planning period \citep{mr1993}.

Research has been done based on the Power-of-Two policy, including
 \cite{LY2003, MR1987, TB2001}. \cite{TB2001} have also noted that
finding the optimal lot sizing policies for stationary demand lot sizing
problems is still an open issue.

\cite{LM1994} presented a fully polynomial time approximation
scheme (FPTAS) for the Strict PJRP model with fixed base. Later, \cite{S20012} presented a quasi-polynomial-time approximation scheme (QPTAS),
which shows that the problem is most likely not $\mathcal{APX}$-hard. In
addition, an efficient polynomial time approximation scheme (EPTAS) for JRP with
finite time horizon and stationary demand was presented by \cite{NS2013}.

This problem was researched in many other different setups, such as JRP under
resource constraints \citep{G1975,K2000,MC2006}, minimum order quantities \citep{PD2006a}, and non-stationary holding cost \citep{LRS2006,NS2009,L2008}.

\textbf{General PJRP.} \cite{PD2005} pointed out that adding
the correction factor leads to a completely different problem, at least in
terms of exact solvability. \cite{PD2004} show that changing
the model from Strict PJRP to PJRP significantly changes the joint
replenishment cycles and the commodities replenishment cycles. The difference
in solvability is evidenced by the sheer number of decision variables. In the
Strict PJRP, all commodities' cycle times are simple functions of the joint
replenishment cycle time. Thus, there is actually only a single decision
variable. However, this is not the case with the PJRP where we have $n$
decision variables, one for each commodity. In practice,
Strict PJRP is much less common than PJRP as it involves paying for empty
deliveries. Strict PJRP may occur only if there is a binding contract with a
delivery company. Although such a binding contract may decrease the cost of
the joint replenishment significantly, it usually limits the flexibility of
choosing the joint replenishment cycles. \cite{ST2011} presented a polynomial time approximation scheme (PTAS) for the
PJRP case.

\textbf{Finite horizon.} Several heuristics were designed to deal with the
finite horizon model. Most of the finite time heuristics assume variable
demands and run-in time $\Omega\left(T\right) $\textbf{\ }\citep{LRS2006,J90}. \cite{ST2011} presented a polynomial-time
$\sqrt{9/8}$-approximation algorithm for the JRP with dynamic policies and
finite horizon. As the time horizon $T$ increases, the ratio converges to
$\sqrt{9/8}$. \cite{ST2011} also presented an FPTAS for the
Strict PJRP case with no fixed base and a finite time horizon.

\section{\label{Model Formulation}Model Formulation}

We consider the case of an infinite time horizon, and a system composed of
several commodities, for each of which there is an external stationary demand.
The demand has to be satisfied in a timely fashion so as to prevent delays. Backlogging and lost sales are
not allowed. Each commodity incurs a fixed ordering cost for each time at
which an order of the commodity is placed, as well as a linear inventory
holding cost for each time unit (referred to as a period) a unit of commodity remains in storage. In
addition, a joint ordering cost is incurred for each time where one or
more orders are placed. We use the following notations, were the units are
given in square brackets:%
\begin{align*}
&  N-\text{ Number of commodities in the system,}\left[units\right]
\text{.}\\
&  \lambda_{c}-\text{Demand rate for commodity}c\text{per period,}\left[
\frac{units}{period}\right]\text{.}\\
&  h_{c}-\text{Holding cost for commodity }c\text{ per period,}\left[
\frac{\$}{units\cdot period}\right]\text{.}\\
&  K_{c}-\text{Fixed ordering cost for commodity }c,\text{}\left[\$\right]
\text{.}\\
&  K_{0}-\text{Fixed joint ordering cost,}\left[\$\right]\text{.}%
\end{align*}
The objective is to find an ordering cycle time, $t_{c} \in \mathbb{R}$, for each
commodity, $c$, so as to minimize the periodic sum of ordering and holding
costs of all commodities.

The simple model, in which there is only a single commodity, is known as the
Economic Order Quantity (EOQ). While examining commodity $c,$ we define its
\emph{standalone problem} as the optimal ordering quantity problem for a
single commodity, $c$, with no joint setup cost and an infinite horizon. The
\emph{standalone problem }is a simple EOQ problem.

The EOQ model assumes without loss of generality that there is no on-hand
inventory at time $0$. Shortage is not allowed, so we must place an order at
time $0$. The average periodic cost, as a function of the cycle time $t_{c},$
denoted by $g\left(t_{c}\right),$ is given by
\begin{equation}
g\left(t_{c}\right) =\frac{K_{c}}{t_{c}}+\lambda_{c}h_{c}\frac{t_{c}}{2},
\label{EOQ formula}%
\end{equation}
and the optimal cycle time for $g\left(t_{c}\right)$, denoted by
$t_{c}^{\ast},$ is%
\begin{equation}
t_{c}^{\ast}=\sqrt{\frac{2K_{c}}{h_{c}\lambda_{c}}}.
\label{EOQ optima solution}%
\end{equation}


See full elaboration and additional analysis in \cite{N2001} and \cite{Z2000}.

\section{$\mathcal{NP}$-Hardness proof}
The $\mathcal{NP}$-Hardness proof of the CPJRP is based on the $\mathcal{NP}$-Hardness proof of the DPJRP in \cite{CY2018}. In this proof, we take advantage of the instance of the DPJRP used in \cite{CY2018} with slight changes in the instance parameters and without the integrality constraint that defines the DPJRP. We show that the instance we construct for the CPJRP is as hard as the \emph{3SAT} instance from which \cite{CY2018} constructed their 
instance, and thus, Strongly $\mathcal{NP}$-hard. 

The \emph{$3SAT$} is defined as follows:
\begin{definition}
Given a logical expression, $\varphi,$ in a Conjunctive Normal Form
(CNF)\footnote{An expression that is a conjunction of clauses, where each
clause is a disjunction of literals.} with $m$ clauses and $n$ variables,
$x_{1},...,x_{n},$ where each clause contains exactly 3 literals, is there a
feasible assignment to the variables such that each clause contains at least
one true literal?
\end{definition}
Before we continue with the $\mathcal{NP}$-hardness proof, we would like to show a schematic sketch of the proof in \citet{CY2018}, as in this paper we meticulously show that each step in this proof holds in the continuous environment for the instance we built.

\subsection{Proof sketch} \label{Sec: Proof sketch}
In this section, we explain the steps taken in \cite{CY2018}. For each step we explain the adjustments required for a continuous environment.
\begin{enumerate}
\item \underline{Polynomial time reduction} -- Given the instance of \emph{3SAT}, three sets of commodities were constructed,
\begin{itemize}
\item \emph{Constant} commodities - commodities, denoted $c^y_i$, that in a discrete environment would be ordered at $t^{*}_{c^y_i}$ regardless of the other commodities.
\item  \emph{Variable} commodities - commodities, denoted $c^x_i$, whose $t_{c^x_i}$ may change according to other commodities' order pattern. Each commodity in this set is associated with one variable in the original \emph{3SAT} problem. The optimal cycle time of $c_i^x \in \emph{Variables}$ is one of two unique prime numbers $\underline{p}_i$ and $\overline{p}_i=\underline{p}_i+b_i$. Each such option is associated with either the variable or its negation in the original $\emph{3SAT}$ problem: $\underline{p}_i$ is associated with $x_i = False$ and $\overline{p}_i$ with $x_i = True$.
\item \emph{Clause} commodities - commodities, denoted $c^z_i$, which, just like the \emph{Constants}, in a discrete environment would be ordered at $t^{*}_{c^z_i}$ regardless of the other commodities. However, $t^{*}_{c^z_i}$ is associated with a clause in the original \emph{3SAT}. $t^{*}_{c^z_i}$
is the product of the primes associated with the clause's three literals.
Accordingly, if the cycle time of at least one of the relevant variable
commodities were set to the cycle time associated with the right literal, the
cycle time of the \emph{Clauses} commodity would be synchronized with it. For example, given a clause $\left (x_3 \cup \overline{x}_6 \cup x_9\right )$ we create a commodity $c^z_i$ with $t^{*}_{c^z_i}=\underline{p}_3 \cdot \overline{p}_6 \cdot \underline{p}_9$.

\end{itemize}
We denote the instance created for the JRP problem by $\Gamma$ (both for the DPJRP and the CPJRP).
In our reduction, we use the same three sets, but adjust the holding and setup costs of commodities \emph{Constants} and \emph{Clauses}.
\item \underline{\emph{Constants} and \emph{Clauses} cycle times} -- \cite{CY2018} showed that the optimal cycle time of \emph{Constants} and \emph{Clauses} are set to $t^{*}_{c^y_i}$ and $t^{*}_{c^z_i}$, respectively, regardless of other commodities.\\
In a continuous environment, the optimal cycle time would always be influenced by other commodities. However, we show that it is restricted to a very narrow time interval. Moreover, we show that in an optimal solution all \emph{Constants} and \emph{Clauses} are synchronized in a way that resembles a discrete environment. That is, for any optimal solution there exists a seed $\beta$, for which in that optimal solution $\forall c^y_i \in \emph{Constants}: t^{*}_{c^y_i}=\beta t^{*}_{c^y_i}$ and $\forall c^z_i \in \emph{Clauses}:t^{*}_{c^z_i}=\beta t^{*}_{c^z_i}$.
\item \underline{\emph{Variables} cycle time} -- \cite{CY2018} showed that the optimal cycle time of variable $t_{c^x_i}$ is one of two unique prime numbers $\underline{p}_i$ and $\overline{p}_i=\underline{p}_i+b_i$.
We show that in a continuous environment, the same applies with a small change. The optimal cycle time for a \emph{Variables} commodity applies $c_i^x \in \left \{\beta \underline{p}_i,\beta \overline{p}_i \right\}$.

\item \underline{Optimal solution} -- \cite{CY2018} show that if there is a solution to the original \emph{3SAT} problem, in an optimal solution of $\Gamma$, the cycle times of \emph{Variables} are set to synchronize with all of the commodities of type \emph{Clauses}.\\
To do so, they formulated the total cost of solution $S$, denoted by
$TC\left( S\right)$, as a sum of three cost functions: The first cost
function, $TC_{\emph{Constants}}\left( S\right) $, sums all the costs
that are associated with the commodities $c^y_i\in\emph{Constants}$,
including all the joint replenishment costs induced by $c^y_i\in\emph{Constants}$. The second cost function, $TC_{\emph{Variables}}\left(
S\right)$, sums all the costs that are associated with the commodities
$c_{i}^{x}\in\emph{Variables}$, including only the marginal addition to the joint replenishment costs induced by $c^x_i\in\emph{Variables}$, assuming all the joint replenishment costs induced by $c^y_i\in\emph{Constants}$ are already paid for. The third
ךcost function, $TC_{\emph{Clauses}}\left( S\right)$, sums all the costs
that are associated with the commodities $c_i^z\in\emph{Clauses}$, including only the marginal addition to the joint replenishment costs induced by $c^z_i\in\emph{Clauses}$ assuming all the joint replenishment costs induced by $c^y_i\in\emph{Constants}$ and by $c_{i}^{x}\in\emph{Variables}$ are already paid for.\\
To show the equivalence to the \emph{3SAT}, they proved the following two steps:
\begin{enumerate}
\item Ignoring commodities of type \emph{Clauses}, for each $c_i^x \in Variables$ the marginal cost of setting $t_{c_i^x}=\underline{p}_i$ is lower than setting it to $t_{c_i^x}=\overline{p}_i$; thus showing that setting $t_{c_i^x}=\underline{p}_i$ for all $c_i^x \in \emph{Variables}$ gives a lower bound on the marginal cost of \emph{Variables}, denoted $LB(TC_{\emph{Variables}})$ and setting $t_{c_i^x}=\overline{p}_i$ for all $c_i^x \in \emph{Variables}$ gives an upper bound on the marginal cost of \emph{Variables}, denoted $UB(TC_{\emph{Variables}})$.
\item Not synchronizing even one commodity of type \emph{Clauses} with at least one of the commodities of type \emph{Variables} costs more than $UB(TC_{\emph{Variables}})-LB(TC_{\emph{Variables}})$
\end{enumerate}
by that proving that an optimal solution to $\Gamma$ is equivalent to at least one true assignment to the original \emph{3SAT} instance, if such exists.\\
We show that both proofs still hold even when we set $t_{c^y_i} \in \left \{\beta t^{*}_{c^y_i},\beta t^{*}_{c^z_i}\right\}$ and $t^*_{c^y_i}=\beta t_{c^y_i}$ instead of $t_{c^y_i} \in \left \{t^{*}_{c^y_i},t^{*}_{c^z_i}\right\}$ and $t^*_{c^y_i}= t_{c^y_i}$, respectively.
\end{enumerate}

In the following we use the titles of the steps in \secref{Sec: Proof sketch} as subsections associated with each step.

\subsection{Polynomial time reduction}

The total cost of types $c_i^y \in \emph{Constants}$ and $c_i^z \in \emph{Clauses}$ commodities in \cite{CY2018} is a constant. We change the cost parameters of these commodities to adjust them to a continuous environment with as little change to their respective optimal cycle times as possible. Accordingly, we create a commodity for which the optimal solution is $\beta t^{*}_{c^y_i} \left(\beta t^*_{c^z_i}\right)$, where the seed $\beta$ is close to 1, regardless of the cycle times of the other commodities in the problem.

In \cite{CY2018}, each commodity $c^y_i\in\emph{Constants} \ \left (c^z_i\in\emph{Clauses}\right )$ is associated with a time $t^*_{c^y_i} \ \left(t^*_{c^z_i}\right)$. This cycle time is the optimal cycle time for that commodity regardless of the solution to other commodities. For our reduction, we denote an auxiliary parameter
\begin{equation}
    \label{eq: delta} \delta=\frac{1}{6n \left ( \overline{p}_n \right)^{6} }.
\end{equation}

We set the holding cost ($h_c$) and ordering cost ($K_c$) for each commodity, $c^y_i\in\emph{Constants}$ as follows:%
\begin{align}
h_{c^y_i}  &  =\frac{1}{(\delta^2+2\delta)(t^*_{c^y_i})^2},\label{DEF const h}\\
K_{c^y_i}  &  =\frac{1}{(\delta^2+2\delta)}.\label{DEF const K}%
\end{align}
Similarly, for each commodity, $c^z_i\in\emph{Clauses}$:
\begin{align*}
h_{c^z_i}  &  =\frac{1}{(\delta^2+2\delta)(t^*_{c^z_i})^2},
K_{c^z_i}  &  =\frac{1}{(\delta^2+2\delta)}.
\end{align*}
We make no changes in the holding and setup costs of $c_i^x\in\emph{Variables}$:
\begin{align}
h_{c_{i}^{x}}  &  =\alpha_{c}\frac{\underline{p}_{i}^{2}-b_{i}^{2}}%
{\underline{p}_{i}\left( \underline{p}_{i}+\frac{b_{i}}{2}\right)
\frac{b_{i}}{2}},\label{DEF vars h}\\
K_{c_{i}^{x}}  &  =h_{c_{i}^{x}}\cdot\underline{p}_{i}\left(  \underline
{p}_{i}+b_{i}\right) -\frac{\underline{p}_{i}+b_{i}}{\underline{p}_{i}%
+b_{i}-1}\alpha_{c}\overline{\alpha}_{v}, \label{DEF vars K}%
\end{align}
where $\alpha_{c},\ \overline{\alpha}_{v},\ \underline{p}_{i}$, and $b_i$ are constants taken from \cite{CY2018}, such that $\underline{p}_{i}$ and $\overline{p}_{i}$ are prime numbers that are unique to commodity $c_i^x$. Without loss of generality, we assume that the primes are sorted in ascending order, which makes $\underline{p}_n=\max \left\{\underline{p}_i \right \}$. Moreover, according to \cite{CY2018}, $b_i$ is small enough so that $\forall{i}:\overline{p}_i=\underline{p}_i+b_i<\underline{p}_{i+1}.$  \\
Just like \cite{CY2018}, We set the joint ordering cost to be:
\begin{equation}
K_{0}=1, \label{DEF K_0}%
\end{equation}
and the demand for each commodity per time unit to be:
\begin{equation}
\forall c:\lambda_c=2. \label{DEF lambda}
\end{equation}

\subsection{\emph{Constants} and \emph{Clauses} cycle times} \label{Sec: Const and Claus}
In this section, we refer only to $c_i^y\in\emph{Constants}$ and note that everything applies to $c_i^z\in\emph{Clauses}$ as well due to their similar cost functions.\\
For each commodity $c_i^y\in\emph{Constants}$, we define two EOQ
problems. In the first EOQ problem, denoted $\theta_{1},$ we define:
$h_{1}=h_{c_i^y}$ and $K_{1}=K_{c_i^y}$. The solution for this
problem defines a lower bound on the marginal average periodic cost of
commodity ${c_i^y}$. This problem is in fact the standalone cost of commodity ${c_i^y}$. In the second EOQ problem, denoted $\theta_{2},$ we define: $h_{2}=h_{c_i^y}$ and $K_{2}=K_{c_i^y}+K_{0}$. That is, we pay $K_{0}$ for
each order of commodity ${c_i^y}$. The solution for this problem defines an
upper bound on the marginal average periodic cost of commodity ${c_i^y}$.
Let us define the optimal solution for $\theta_i$ by $t_i$ for $i=1,2$
Substituting for $h_{c_i^y}$, $K_{c_i^y}$, and $\lambda_{c_i^y}$ using Eqs. $\left(
\ref{DEF const h}\right)$, $\left( \ref{DEF const K}\right)$, and (\ref{DEF lambda}) into Eq. $\left( \ref{EOQ optima solution}\right),$ we get:%
\begin{equation}
\label{SAC} t_1=t^*_{c_i^y}.
\end{equation}
Similarly, substituting for $h_c$, $K_c$, $k_0$, and $\lambda_c$ using Eqs. 
\eqref{DEF const h}, \eqref{DEF const K}, \eqref {DEF K_0}, and \eqref{DEF lambda} into Eq. \eqref{EOQ optima solution}, we get:%
\[
t_2=(1+\delta) t^*_{c_i^y}.
\]
\begin{corollary}
\label{Corollary constants cycle time} In any optimal solution to $\Gamma,$ $t^*_{c_i^y}\leq t_{c_i^y}\leq (1+\delta) t^*_{c_i^y}$ for any $c\in\emph{Constants}$\emph{.}
\end{corollary}

Let us consider a non-discrete cycle time $t_{c_i^y}$ to be the product of a discrete number $t^{*}_{c_i^y}$ and a seed $\beta_{c_i^y}$, i.e., $t_{c_i^y} = t^{*}_{c_i^y} \cdot \beta_{c_i^y}$, where $1 \leq \beta_{c_i^y} \leq 1+\delta$.

For the analysis that follows we make use of two functions that quantify the average number of joint replenishments per time unit. Let us assume that all the commodities are ordered at time $0$. If the ratios between all $\beta_{c_i^y}$ are rational, there is a time $T$, which is the least common multiplier of all cycle times, in which all commodities will be ordered together again. Hence, it is sufficient to find the average number of joint replenishments within $T$ only. If, on the other hand, the ratios between some $\beta_{c_i^y}$ are irrational, then they will never be ordered together again. In that case we can calculate the average number of joint replenishments per period for each set of commodities with a rational ratio between them separately and sum these averages.
Accordingly, we refer to a finite time horizon $T$.
Each cycle time $t_{c_i^y}$ represents a series of orders $t_{c_i^y}$ time apart. Let us denote such a series by $F_{t_{c_i^y}}$, where $F_{t_{c_i^y}}=\left \{ t_{c_i^y},2t_{c_i^y},...,T\right\}$. A union over all these series gives us the set of all joint orders in time horizon $T$. The average number of joint replenishments per time period is the cardinality of the set of order points within the time horizon divided by $T$. That is:
\begin{equation*}
\frac{\left | \bigcup\limits_{t_{c_i^y} \in S} F_{t_{c_i^y}} \right |}{T},
\end{equation*}
where the absolute value (within the vertical bars) defines the cardinality of a set.\\
Comparing the average number of joint replenishments, we define two functions over general sets of time points $F_1,F_2,..,F_m$ (all bounded by $T$).

\begin{definition}[Union of Joint Replenishment ($UJR$)]Function
$UJR:F^m \mapsto \mathbb{R^{+}}$, where $m\leq n$ is the number of time series unionized, represents the average number of joint replenishments per time unit \emph{where at least one commodity belonging to one of the series $F_i$ for $i=1,...,m$  is ordered}.
\begin{equation*}
UJR(F_1,...,F_m)=\frac{\left | \bigcup\limits_{i=1}^{m} F_i \right |}{T}.
\end{equation*}
\end{definition}
\begin{definition}[Intersection of Joint Replenishment ($IJR$)] Function $IJR:(F^m) \mapsto \mathbb{R^{+}}$ represents the average number of joint replenishments per time unit \emph{where at least one commodity from each series $F_i$ for $i=1,...,m$ is ordered}.
\begin{equation*}
IJR(F_1,...,F_m)=\frac{\left | \bigcap\limits_{i=1}^{m} F_i \right |}{T}.
\end{equation*}
\end{definition}

For only two cycle times, $IJR$ could be calculated explicitly by using the least common multipliers (LCM). LCM could be used to define a tight upper bound on the frequency of intersection between two (or more) arithmetic sequences of numbers. This is an upper bound because it reflects the frequency of the intersection between the sequences that are in-phase, and if they're out of phase, the number of intersections would be zero (consider for example a commodity ordered every even time period and a commodity ordered every odd time period). However, an optimal solution will always strive to synchronize commodities orders to minimize the number of joint replenishments. Hence,
\begin{equation}
\label{LCM2IJR}
IJR(F_{t_1},F_{t_2})=\frac{1}{LCM(t_1,t_2)}.
\end{equation}
Note that LCM satisfies:
\begin{equation*}
\label{LCM}
LCM(a\cdot b,a \cdot c)=a\cdot LCM(b,c).
\end{equation*}

$UJR$ and $IJR$ are in fact the cardinality of union and intersection of sets normalized by a constant $T$. Accordingly, $UJR$ and $IJR$ hold the cardinality characteristics of union and intersection, some of which we make use of in our proof.
\begin{enumerate}
\item $UJR(F_1,F_2)=UJR(F_1)+UJR(F_2)-IJR(F_1,F_2)$
\item $IJR(F_1,F_2,F_3)\le IJR(F_1,F_2)$
\item $IJR(F_1,(F_2\cup F_3))=IJR(F_1,F2)+IJR(F_1,F_3)-IJR(F_1,F_2,F_3)$
\item $IJR(F_1,F_2) \leq IJR(F_3,F_4)$ for $F_1 \subseteq F_3$ and $F_2 \subseteq F_4$.
\end{enumerate}

Note that we do not always calculate $UJR$ or $IJR$ explicitly, but rather make use of the bounds on those functions in the analysis that follows.

\begin{theorem}
\label{theorem:same seeds}
In any optimal solution all seeds to the set \emph{Constants} will be identical, i.e., $\exists \beta, \forall {c_i^y} \in \emph{Constants} : \beta_{c_i^y} = \beta$.
\end{theorem}

\begin{proof}
We will prove the Theorem \ref{theorem:same seeds} by contradiction. To those means, we will assume that there exists an optimal solution $S$ that has at least two different seeds ($\beta's$).\\
Let $S$ be such a solution, i.e., an optimal solution where the commodities are centered around $k$ different seeds$ \beta_{1} \le \beta_{2} \le \dots \le \beta_{k}$. We denote a subset of the commodities that are centered around the same seed $\beta_{i}$ in $S$ as $A_{i}$. That is, ${c_i^y} \in A_{i}$ iff $t_{c_i^y} = t^*_{c_i^y}\beta_{i}$. In what follows, we compare solution $S$ with some solution $S'$, which is similar to $S$ with one difference: the cycle times of the commodities in set $A_{2}$ in $S'$ are changed to be centered around seed $\beta_{1}$ instead of $\beta_{2}$, which means that $\forall {c_i^y}\in A_{2}:t_{c_i^y} = t_{c_i^y}^{*}\beta_{1}$. We proceed to show that $S'$ is cheaper than $S$, which contradicts the optimality of $S$.

We separate the cost functions of $S$ and $S'$ to the total standalone costs and the joint replenishment costs and compare them separately.\\
According to Eq. (\ref{SAC}), the optimal solution to $\theta_1$, (which is the standalone cost of commodity ${c_i^y} \in \emph{Constants}$ is: $t_{c_i^y}=t^*_{c_i^y}$. Since the standalone cost is convex, the closer the seed $\beta_c$ is to 1, the cheaper the standalone cost is. The difference between $S$ and $S'$ is in set $A_2$ whose seed was changed from $\beta_2$ to $\beta_1$. Since $1\leq\beta_1<\beta_2$ by definition, we have the following corollary:
\begin{corollary}
The cumulative standalone cost of all commodities in $S'$ is smaller than that in $S$.
\end{corollary}

Next, we show that the cumulative joint replenishment cost in $S'$ is smaller than in $S$.

The cost of joint replenishment is linearly dependent on the average number of joint replenishments per time unit, 
or, in terms of $UJR$ functions, we'd like to show that $UJR_{S'} - UJR_{S} \leq 0$, where $UJR_{x}$ is short for the $UJR$ function on all the commodities of an arbitrary solution $x$.
Let us represent the time series of commodities that share the same seed, $\beta_i$, as a pair $(A_i,\beta_i)$. That is, $(A_i,\beta_i)=\bigcup\limits_{c_I^y \in A_i} F_{\beta_i \cdot t^*_{c_i^y}}.$
The expressions for the $UJR_{S'}$ and $UJR_{S}$, respectively, are:
{\begin{multline}
\label{EQ:UJRS'}
UJR_{S'} = UJR((A_1,\beta_1),(A_2,\beta_1),(A_3,\beta_3),...,(A_n,\beta_n)) =
\\ =
\underbrace{UJR((A_1,\beta_1),(A_3,\beta_3),(A_4,\beta_4),...,(A_n,\beta_n))}_{S'_1}+\underbrace{UJR((A_2,\beta_1))}_{S'_2}\\
- \underbrace{IJR((A_2,\beta_1), \{(A_1,\beta_1)\cup (A_3,\beta_3)\cup (A_4,\beta_4)\cup ...\cup (A_n,\beta_n)\})}_{S'_3},
\end{multline}}
\leavevmode
{\begin{multline}
\label{EQ:UJRS}
UJR_{S} = UJR((A_1,\beta_1),(A_2,\beta_2),(A_3,\beta_3),...,(A_n,\beta_n)) =
\\ = \underbrace{UJR((A_1,\beta_1),(A_3,\beta_3),(A_4,\beta_4),...,(A_n,\beta_n))}_{S_1}+\underbrace{UJR((A_2,\beta_2))}_{S_2}\\
-\underbrace{IJR((A_2,\beta_2), \{(A_1,\beta_1)\cup (A_3,\beta_3)\cup (A_4,\beta_4)\cup ...\cup (A_n,\beta_n)\})}_{S_3}.
\end{multline}}
To simplify the tractability, we refer to the elements of $UJR_{S'}$ and $UJR_{S}$ in Eqs. (\ref{EQ:UJRS'}) and (\ref{EQ:UJRS}) by their underlined notations, $S'_i$ and $S_i$ for $i=1,2,3.$
In what follows, we examine the difference $UJR_{S'} - UJR_{S}=\sum_{i=1}^3 \left(S'_i-S_i \right )$, and show that it is negative.

$S'_1=S_1$ and can both be omitted.

$S'_2-S_2$ is positive, since in $S'$ we allow the commodities in $A_2$ to be ordered more often.
However, Since $(A_2,\beta_2))$ and $(A_2,\beta_1)$ represent the same time series multiplied by different constants, we have $UJR((A_2,\beta_2))=UJR((A_2,\beta_1))\frac{\beta_2}{\beta_1}$. Therefore,
\begin{equation}
\label{EQ B_E}
S'_2-S_2=UJR((A_2,\beta_1)) \frac{\beta_{1} - \beta_{2}}{\beta_{1}\beta_{2}} \leq UJR((A_2,\beta_1)) \left (\beta_{1} - \beta_{2} \right )\leq \delta \cdot UJR((A_2,\beta_1)) \leq \delta,
\end{equation}
where the first inequality comes from both $\beta_1 \geq 1$ and $\beta_2 >1$, the second inequality comes from the fact that $1 \leq \beta_{1}<\beta_{2} \leq 1+\delta$ ,and the last inequality comes from the fact that all the elements in $A_2$ are integers and $\beta_1 \geq 1$ and thus, there could not be more than one order per period.



Now we shall decompose $S'_3$ and $S_3$ parts of Eqs. (\ref{EQ:UJRS'}) and (\ref{EQ:UJRS}), respectively, by the rule introduced in the third characteristic of $UJR$ and $IJR$ we discussed above. Also, remember that an upper bound on $S_3$ and a lower bound on $S'_3$ are good enough:
\begin{multline}
\label{eq:C}
S'_3 = IJR((A_2,\beta_1), \{(A_1,\beta_1),(A_3,\beta_3),(A_4,\beta_4),...,(A_n,\beta_n)\})  \\ = IJR((A_2,\beta_1), (A_1,\beta_1)) + IJR( (A_2,\beta_1), \{(A_3,\beta_3),(A_4,\beta_4),...,(A_n,\beta_n)\}) \\ -
IJR( (A_2,\beta_1), (A_1, \beta_1), \{(A_3,\beta_3),(A_4,\beta_4),...,(A_n,\beta_n)\}) \\ \geq IJR((A_2,\beta_1), (A_1,\beta_1)),
\end{multline}

where the inequality follows the second characteristic of $UJR$ and $IJR$ we discussed above.

An upper bound on $S_3$ could be found by decomposing it according to the third characteristic of the $UJR$ and $IJR$ functions, and ignoring the negative elements. That is:
\begin{multline}
\label{eq:F}
S_3 = IJR((A_2,\beta_2), \{(A_1,\beta_1),(A_3,\beta_3),(A_4,\beta_4),...,(A_n,\beta_n)\})  \\  \leq \sum_{i=1,3,4,...,n} IJR((A_2,\beta_2), (A_i,\beta_i)).
\end{multline}

Considering Eqs. (\ref{EQ:UJRS'})-(\ref{eq:F}), we have:

\begin{equation}
\label{eq:simplediff}
\begin{gathered}
UJR_{S'} - UJR_{S}=-IJR((A_2,\beta_1), (A_1,\beta_1))
+ \delta+ \sum_{i=1,3,4,...,n} IJR((A_2,\beta_2), (A_i,\beta_i)).
\end{gathered}
\end{equation}

We'll divide the expression into its composing parts and bound each of them separately.
Note that $(A_2,\beta_1)$ and $(A_1,\beta_1)$ share the same seed. This means we can create a lower bound on their $IJR$ that is independent of $\delta$. The worst case scenario would be both groups consisting of exactly one commodity each (since adding any commodities would only increase the $IJR$), and those two commodities having the longest cycle time.
According to \cite{CY2018} the longest cycle time is associated with commodities of type $\emph{Clauses}$, where for ${c^z_i} \in \emph{Clauses}$, $t^{*}_{c^z_i}$
is the product of the primes associated with the clause's three literals. Hence, $\forall{i}:t^{*}_{c^z_i}<\max \left \{ \left (\overline{p}_i \right)^3 \right \}=\left (\overline{p}_n \right)^3$. This means, 
\begin{equation}
\label{Eq tmax}
IJR((A_2,\beta_1), (A_1,\beta_1))>\frac{1}{\beta_1\left (\overline{p}_n \right)^6}>\frac{1}{(1+\delta)\left (\overline{p}_n \right)^6}>\frac{1}{2\left (\overline{p}_n \right)^6}.
\end{equation}

To bound a term $IJR((A_2,\beta_2), (A_i,\beta_i))$ we would first need to prove an additional lemma,
\begin{lemma}
\label{lemma:IJRbound} 
$IJR((A_{i}, \beta_{i}), (A_{j}, \beta_{j})) \leq \delta$
\end{lemma}

\begin{proof}
We distinguish between two cases:
\begin{itemize}
\item $\frac{\beta_{j}}{\beta_{i}}$ is irrational. \\
In this case, $IJR((A_{i}, \beta_{i}), (A_{j}, \beta_{j})) \rightarrow 0$, since there will be a maximum of one instance of joint replenishment in an infinite time horizon, and the replenishments will never coincide again.
\item $\frac{\beta_{j}}{\beta_{i}}$ is rational. \\
According to the fourth characteristic of $UJR$ and $IJR$, we have $IJR((A_{i}, \beta_{i}), (A_{j}, \beta_{j})) \leq IJR(F_{\beta_{i}}, F_{\beta_{j}})$, as the series covered by a single commodity with a cycle time of $\beta_{i}$ is a superset of the series $(A_{i}, \beta_{i})$.
Hence, we consider $IJR(F_{\beta_{i}}, F_{\beta_{j}})$ as an upper bound on $IJR((A_{i}, \beta_{i}), (A_{j}, \beta_{j}))$.
Since $\frac{\beta_{j}}{\beta_{i}}$ is rational, we can express it as an irreducible fraction $1 + \frac{q}{r}$, where $q$ and $r$ are positive integers. \\
Using the $LCM$ function to quantify $IJR(F_{\beta_{i}}, F_{\beta_{j}})$, we can observe that
\begin{equation*}
LCM\left (\beta_i,\beta_{j} \right )=\beta_i LCM\left ( 1, 1+\frac{q}{r} \right ) = \frac{\beta_i}{r} LCM(r, r+q)=\frac{\beta_i}{r}(r(r+q))=\beta_i(r+q),
\end{equation*}
where the third equality is true because $r$ and $q$ share no common multipliers.\\
According to Eq. \eqref{LCM2IJR}, \[IJR(F_{\beta_{i}}, F_{\beta_{j}}) = \frac{1}{\beta_i(r+q)}\le \frac{1}{r+q}.\]
Now, let us consider the largest integer $a$, such that $\frac{1}{a} \geq \delta$
Since $\frac{\beta_j}{\beta_i}=1+\frac{q}{r}<1+\delta \leq 1+\frac{1}{a}$, we have $a\cdot q <r$ and thus,
\[
IJR((A_{i}, \beta_{i}), (A_{j}, \beta_{j}))\leq IJR(F_{\beta_{i}}, F_{\beta_{j}}) \leq \frac{1}{(1+a)q} \leq \frac{1}{a+1} < \delta,
\]
where the last inequality comes from the definition of $a$ as the largest integer for which $\delta \leq \frac{1}{a}$
\end{itemize}
This completes the proof of Lemma \ref{lemma:IJRbound}.

\end{proof}
\leavevmode

According to Lemma \ref{lemma:IJRbound}, $\sum_{i=1,3,4,...,n} IJR((A_2,\beta_2), ((A_j,\beta_j))\leq (n-1)\delta$.

Substituting this along with Eqs. \eqref{Eq tmax} and \eqref{eq: delta} into Eq. \eqref{eq:simplediff} we get:
\[
UJR_{S'} - UJR_{S} \leq \frac{-1}{2\left (\overline{p}_n \right)^6} + n\delta = \frac{-3n}{6n\left (\overline{p}_n \right)^6} + \frac{n}{6n\left (\overline{p}_n \right)^6} =\frac{-1}{3\left (\overline{p}_n \right)^6}< 0.
\]
This contradicts the optimality assumption of $S$ and proves Theorem \ref{theorem:same seeds}.
\end{proof}


\subsection{\emph{Variables} cycle times} \label{Sec: Var}
Theorem \ref{theorem:same seeds} establishes that the processing of all commodities of type \emph{Constants} share the same seed $\beta$. Normalizing the time by $\beta$, all these commodities are ordered at the same discrete cycle times as in \cite{CY2018}.
Next, we show that the cycle times of commodities of type \emph{Variables} are also restricted to the same discrete cycle times (when time is normalized by $\beta$) as in \cite{CY2018}.
\begin{theorem}
\label{theorem: Variables}
In an optimal solution to $\Gamma$, $\forall {c_i^x} \in \emph{Variables}:t_{c^x_i}\in \left \{ \underline{p}_{i},\overline{p}_{i}\right \}.$
\end{theorem}

\begin{proof}
Our proof makes use of \cite{CY2018}, who proved that the cycle time of $c_i^x \in \emph{Variables}$ is one of two unique prime numbers $\underline{p}_{i}$ and $\overline{p}_{i}$. The standalone cost of these commodities, $t^*_{c^x_i}$ satisfies $\underline{p}_{i}<t^*_{c^x_i}<\overline{p}_{i}.$
Their proof was constructed of two parts
\begin{enumerate}
\item The cycle time of $c^x_i \in \emph{Variables}$ satisfies $t_{c^x_i} > \underline{p}_{i}-1$ and $t_{c^x_i} < \overline{p}_{i}+1$.
\item The cost of setting $t_{c^x_i} \in \left \{ \underline{p}_{i},\overline{p}_{i} \right \}$ dominates any other solution where $t_{c^x_i} \in \left (\underline{p}_{i},\overline{p}_{i}\right )$.
\end{enumerate}
The proof of the first part simply showed that just the standalone cost (without paying any joint replenishment costs) of a solution with $t_{c^x_i} \in \left \{ \underline{p}_{i}-1,\overline{p}_{i}+1 \right \}$ is more expensive than the worst case solution for $t_{c^x_i}=\overline{p}_{i}$. This proof is independent of the discrete environment and holds in a continuous environment as well, and due to the convexity of the standalone cost, that means that in an optimal solution of $\Gamma$, $t_{c^x_i} \in \left (\underline{p}_{i}-1,\overline{p}_{i}+1 \right )$.\\
The second part of their proof refers to the very high synchronization between the cycle times of commodities of type \emph{Constants} with both $\underline{p}_{i}$ and $\overline{p}_{i}$.
\cite{CY2018} constructed their $\mathcal{NP}$-hard instance by setting the optimal cycle times of ${c^y_i} \in \emph{Constants}$ to be a multiplication of each of the primes $\underline{p}_{i}$ and $\overline{p}_{i}$ with a very large set of other primes (for each such multiplication, a commodity of type \emph{Constants} was created).
Accordingly, when setting $t_{c^x_i} \in \{ \underline{p}_{i},\overline{p}_{i} \}$ the marginal cost added to the total cost due to joint replenishments is very small.
They show that even when assuming maximum synchronization between the commodities of type \emph{Variables}, for any solution that is not $\underline{p}_{i}$ or $\overline{p}_{i}$ but still within the range $t_{c^x_i} \in \left (\underline{p}_{i},\overline{p}_{i} \right )$,
and assuming the optimal standalone costs for these commodities, the marginal cost incurred by these commodities is still greater than the marginal cost of setting $t_{c^x_i} \in \left \{ \underline{p}_{i},\overline{p}_{i} \right \}$.
Next, we show that their bound on the synchronization between the cycle time assuming $t_{c^x_i} \in \left (\underline{p}_{i},\overline{p}_{i} \right )$ still holds in a continuous setting.
Using our representation of continuous times as a multiplication of a discrete time and a seed. we refer to a continuous cycle time as $t_{c^x_i}\beta_i$. We may assume $1 \leq \beta_i<\frac{1}{t_{c^x_i}}$ (this is enough to allow us to cover any continuous value).\\
Considering two cycle times, $t_{c^x_i}\beta_i$ and $t_{c^x_j}\beta_j$, where $\beta_i\leq \beta_j$ (without loss of generality).
The rate of synchronization of their cycle times $IJR\left (\left( \{t_{c^x_i}\},\beta_i\right),\left(\{t_{c^x_j}\},\beta_j \right) \right )$, is maximized by minimizing $LCM(t_{c^x_i}\beta_i,t_{c^x_j}\beta_j)$. The LCM function is minimized by finding the largest common multiplier of the two numbers. For any seed, the largest common multiplier is itself, and the smaller the seed, the smaller the LCM. Thus, 
\begin{equation*}
    LCM\left (t_{c^x_i}\beta_i,t_{c^x_j}\beta_j \right) \leq LCM\left (t_{c^x_i}\beta_i,t_{c^x_j}\beta_i\right) \leq LCM\left(t_{c^x_i},t_{c^x_j}\right).
\end{equation*}

In other words, a common minimal seed gives us a lower bound to the synchronization of any two numbers. The lower bound used in \cite{CY2018} is in fact a lower bound with a common seed of 1 and thus, the bound holds for a continuous setting as well.

Moreover, it applies also to any cycle time in the continuous ranges $\left (\underline{p}_{i}-1,\underline{p}_{i} \right ]$ and $\left [\overline{p}_{i},\overline{p}_{i}+1 \right )$ not including  $\beta \overline{p}_{i}$.
The upper bound on the marginal cost of a solution with $t_{c_i^x} \in \left \{ \beta \underline{p}_{i},\beta \overline{p}_{i} \right \}$ also holds as all commodities share the same seed $\beta$ and are practically discrete under a time normalized by $\beta$.

To simplify our analysis, we define the function $\Delta_{c}\left(
t_{c}\right)$ that describes the marginal average cost per time unit associated
with commodity $c$'s cycle time, $t_{c}$, to the other
commodities in the system. We denote a lower and an upper bound on any
general function $f$ by $LB\left( f\right) $ and $UB\left(  f\right),$
respectively. Accordingly, $LB\left( \Delta_{c}\left( t_{c}\right)
\right)$ and $UB\left( \Delta_{c}\left( t_{c}\right) \right)$ are
lower and upper bounds on $\Delta_{c}\left( t_{c}\right),$ respectively.
Thus,
the marginal contribution of a commodity $c_i^x \in \emph{Variables}$ to the commodities $c_i^y \in \emph{Constants}$ is composed of its standalone cost given by Eq. \eqref{EOQ formula} and its marginal contribution to the joint replenishment cost, denoted:
\[
jr^{c_i^x}(t_{c_i^x}):=UJR\left( F_{t_{c_i^x}},\left(\emph{Constants},\beta\right)\right)-UJR\left(\left(\emph{Constants},\beta\right)\right).\]
According to the first characteristic of $UJR$ and $IJR$: 
\[jr^{c_i^x}(t_{c_i^x})= UJR\left( F_{t_{c_i^x}}\right)-IJR\left( F_{t_{c_i^x}},\left(\emph{Constants},\beta\right)\right).
\]
According to \cite{CY2018}, for $t_{c_i^x}\neq \beta \underline{p}_{i}$
\begin{equation}
\label{Eq LB for t}
jr^{c_i^x}(t_{c_i^x}) \ge \frac{K_{0}\left( \alpha_{v}\alpha_{n}\right) }{t_{c_i^x}} \ge \frac{K_{0}\left( \alpha_{v}\alpha_{n}\right) }{\beta t_{c_i^x}},
\end{equation}

and for $t_{c_i^x} \in \left \{ \beta \underline{p}_{i}, \beta \overline{p}_{i} \right \}$ 
\begin{equation}
\label{Eq UB for t}
jr^{c_i^x}(t_{c_i^x}) \le \frac{K_{0}\left( \alpha_c \right )  }{\beta t_{c_i^x}},
\end{equation}

where $\alpha_{v},\ \alpha_{n},$ and $\alpha_c$ are constants taken from \cite{CY2018}, which we decided to leave as is for ease of validation.\\
Next, we show that $UB\left( \Delta_{c_{i}^{x}}\left( \beta \underline{p}%
_{i}\right) \right) <LB\left( \Delta_{c_{i}^{x}}\left( \beta (\underline{p}%
_{i}+y)\right)  \right),$ where $y \neq 0$.
Using Eq. \eqref{Eq LB for t} with $t_{c_{i}^{x}}=\beta (\underline{p}_{i}+y)$, Eq. \eqref{Eq UB for t} with $t_{c_{i}^{x}}=\beta \underline{p}_{i}$, and Eq. \eqref{EOQ formula}, we get:
\begin{align}
&  LB\left(  \Delta_{c_{i}^{x}}\left(  \underline{p}_{i}+y\right)  \right)
-UB\left(  \Delta_{c_{i}^{x}}\left(  \underline{p}_{i}\right)  \right)
\nonumber\\
&  =\frac{K_{c_{i}^{x}}+K_{0}\left(  \alpha_{v}\alpha_{n}\right)  }{\beta \left (\underline
{p}_{i}+y \right )}+\beta h_{c_{i}^{x}}\left(  \underline{p}_{i}+y\right)  -\left(
\frac{K_{c_{i}^{x}}}{\beta \underline{p}_{i}}+\beta \underline{p}_{i}h_{c_{i}^{x}}%
+K_{0}\cdot\frac{\alpha_{c}}{\beta \underline{p}_{i}}\right) \nonumber\\
&  =\frac{-yK_{c_{i}^{x}}}{\beta \underline{p}_{i}\left(  \underline{p}%
_{i}+y\right)  }+\beta yh_{c_{i}^{x}}+ K_0 \left ( \frac{\alpha_{v}\alpha_{n}}{\beta \left(\underline{p}_{i}
+y\right)}-\frac{\alpha_{c}}{\beta \underline{p}_{i}} \right ). \label{EQ LB(p+ -1) -UB(p+) with p+}%
\end{align}
Substituting for $h_{c_{i}^{x}},$ $K_{c_{i}^{x}}$ and $K_{0}$ using Eqs.
$\left(  \ref{DEF vars h}\right)  $, $\left(  \ref{DEF vars K}\right)  $ and
$\left(  \ref{DEF K_0}\right)  $ into Eq. $\left(
\ref{EQ LB(p+ -1) -UB(p+) with p+}\right)  $ we get:%
\begin{align*}
&  LB\left(  \Delta_{c_{i}^{x}}\left(  \underline{p}_{i}+y\right)  \right)
-UB\left(  \Delta_{c_{i}^{x}}\left(  \underline{p}_{i}\right)  \right) \\
&  =\frac{-yh_{c_{i}^{x}}\cdot\underline{p}_{i}\left(  \underline{p}_{i}%
+b_{i}\right)  +y\frac{\underline{p}_{i}+b_{i}}{\underline{p}_{i}+b_{i}%
-1}\alpha_{c}\overline{\alpha}_{v}}{\beta \underline{p}_{i}\left(  \underline{p}%
_{i}+y\right)  }+\beta yh_{c_{i}^{x}}+\frac{\alpha_{v}\alpha_{n}}{\beta \left(\underline{p}_{i}+y\right)}-\frac{\alpha_{c}}{\beta \underline{p}_{i}}\\
&  =-yh_{c_{i}^{x}}\left(\frac{\underline{p}+b_{i}}{\beta\left(\underline{p}_{i}+y\right)}-\beta\right)+\frac{y\left(
\underline{p}_{i}+b_{i}\right)  \alpha_{c}\overline{\alpha}_{v}}{\beta\underline
{p}_{i}\left(  \underline{p}_{i}+y\right)  \left(  \underline{p}_{i}%
+b_{i}-1\right)  }+\frac{\alpha_{v}\alpha_{n}}{\beta \left(\underline{p}_{i}+y\right)}-\frac{\alpha_{c}}{\beta \underline{p}_{i}}\\
&  =-\alpha_{c}y\frac{\underline{p}_{i}^{2}-b_{i}^{2}}{\underline{p}%
_{i}\left(  \underline{p}_{i}+\frac{b_{i}}{2}\right)  \frac{b_{i}}{2}}%
\left(\frac{\underline{p}+b_{i}}{\beta\left(\underline{p}_{i}+y\right)}-\beta\right)+\frac{y\left(  \underline{p}_{i}%
+b_{i}\right)  \alpha_{c}\overline{\alpha}_{v}}{\beta \underline{p}_{i}\left(
\underline{p}_{i}+y\right)  \left(  \underline{p}_{i}+b_{i}-1\right)  }%
+\frac{\alpha_{v}\alpha_{n}}{\beta \left(\underline{p}_{i}+y\right)}-\frac{\alpha_{c}}{\beta \underline{p}_{i}}\\
&  =\frac{-\alpha_{c}}{\beta\left(\underline{p}_{i}+y\right)}\left(  y\frac{\underline{p}%
_{i}^{2}-b_{i}^{2}}{\underline{p}_{i}\left(  \underline{p}_{i}+\frac{b_{i}}%
{2}\right)  \frac{b_{i}}{2}}\left(\underline{p}_{i}(1-\beta^2)+ b_{i}-\beta y\right)  -\frac{y\left(
\underline{p}_{i}+b_{i}\right)  \overline{\alpha}_{v}}{\underline{p}%
_{i}\left(  \underline{p}_{i}+b_{i}-1\right)  }+\frac{\left(  \underline
{p}_{i}+y\right)  }{\underline{p}_{i}}\right)  +\frac{\alpha_{v}\alpha_{n}}{\beta \left(\underline{p}_{i}+y\right)}\\
&  >\frac{-\alpha_{c}}{\beta \left(\underline{p}_{i}+y\right)}\left(  \frac{2y\left(\underline{p}_{i}(1-\beta^2)+ b_{i}-\beta y\right)  }{b_{i}}+2\right)  +\frac{\alpha_{v}\alpha_{n}}{\beta \left(\underline{p}_{i}+y\right)}\\
&  >\frac{1}{\beta}\left(\frac{-\alpha_{c}}{\underline{p}_{i}+y}\left(  \frac{2y\left(
b_{i}-y\right)  }{b_{i}}+2\right)  +\frac{\alpha_{v}\alpha_{n}}{\underline{p}%
_{i}+y}\right ).
\end{align*}
The last term is identical to the term for $LB\left( \Delta_{c_{i}^{x}}\left( \underline{p}_{i}+y\right) \right)
-UB\left( \Delta_{c_{i}^{x}}\left( \underline{p}_{i}\right) \right)$ in \cite{CY2018}\footnote {See proof for Claim 3 in Appendix D of \cite{CY2018}}
multiplied by a positive constant $\frac{1}{\beta}$, which does not affect the remainder of the proof in \cite{CY2018}, showing that this term is positive.
\end{proof}

\subsection{Optimal solution}
In this section, we trace the two proofs in \cite{CY2018} mentioned in \secref{Sec: Proof sketch} and show that they are valid for $t_{c_{i}^{x}}\in\left\{ \beta \underline{p}_{i},\beta \overline{p}_{i}\right\} $ and $t_{c^y_i}=\beta t^*_{c^y_i}$ instead of $t_{c_{i}^{x}}\in\left\{ \underline{p}_{i},\overline{p}_{i}\right\} $ and $t_{c^y_i}= t^*_{c^y_i}$, respectively.\\
We denote $\Delta_{c_{i}^{x}}^{TC_{\emph{Variables}}}\left( t_{c_{i}^{x}%
},S\right)$ as the marginal average periodic cost added to the function
$TC_{\emph{Variables}}\left( S\right) $ associated with commodity $c_i^x$'s
cycle time $t_{c_{i}^{x}}$, where $c_{i}^{x}\in\emph{Variables}$\emph{,
}$t_{c_{i}^{x}}\in\left\{ \beta \underline{p}_{i},\beta \overline{p}_{i}\right\}$ and a
solution $S$ that applies the characteristics of an optimal solution in
\secref{Sec: Const and Claus} and \secref{Sec: Var} to the other
commodities in the system. In the next lemma, we formulate bounds on
$\Delta_{c_{i}^{x}}^{TC_{\emph{Variables}}}\left( t_{c_{i}^{x}},S\right)$.

\begin{lemma} \label{Lemma Delta p_i}
For any solution $S$ that satisfies the condition of Theorem \ref{theorem: Variables}, $\Delta_{c_{i}^{x}}^{TC_{\emph{Variables}}}\left( \beta \underline{p}_{i},S\right) <\Delta_{c_{i}^{x}}^{TC_{\emph{Variables}}}\left(  \beta \overline{p}_{i},S\right)$
\end{lemma}
\begin{proof}
The function $\Delta_{c_{i}^{x}}^{TC_{\emph{Variables}}}\left(t_{c_i^x},S\right)$ is composed of two parts: the standalone cost of $c_i^x$, denoted $SAC^{c_i^x}(t_{c_i^x})$ and the marginal contribution to the joint replenishment cost, $jr^{c_i^x}(t_{c_i^x})$.
Thus:
\begin{equation} \label{SAC def}
\Delta_{c_{i}^{x}}^{TC_{\emph{Variables}}}\left(t_{c_i^x},S\right)=SAC^{c_i^x}(t_{c_i^x})+jr^{c_i^x}(t_{c_i^x}).
\end{equation}
We base our proof on \cite{CY2018},\footnote {See proof for Claim 4 in Appendix E of \cite{CY2018}.} which states that $\Delta_{c_{i}^{x}}^{TC_{\emph{Variables}}}\left(\underline{p}_{i},S\right) <\Delta_{c_{i}^{x}}^{TC_{\emph{Variables}}}\left(\overline{p}_{i},S\right)$.
Since for $c_i^x$, $\underline{p}_i <\beta \underline{p}_i < t^*_{c_i^x} <\overline{p}_i <\beta \overline{p}_i$, and due to the convexity of the cost function, we may infer that the standalone costs of $c_i^x$ satisfies: \begin{equation} \label{SAC underline p_i}
SAC^{c_i^x}(\underline{p}_i) \ge SAC^{c_i^x}(\beta \underline{p}_i),
\end{equation}
and
\begin{equation} \label{SAC overline p_i}
SAC^{c_i^x}(\overline{p}_i) \le SAC^{c_i^x}(\beta \overline{p}_i)
\end{equation}.

Assuming all commodities share the same $\beta$, $jr^{c_i^x}(t_{c_i^x})$ is linearly increasing in $\frac{1}{\beta}$.
\cite{CY2018} showed that $jr^{c_i^x}(\underline{p}_i)-jr^{c_i^x}(\overline{p}_i)>0$. Hence,
\begin{equation} \label{jr p_i}
jr^{c_i^x}(\underline{p}_i)-jr^{c_i^x}(\overline{p}_i) \ge \frac{1}{\beta} \left (
jr^{c_i^x}(\underline{p}_i)-jr^{c_i^x}(\overline{p}_i) \right )= jr^{c_i^x}(\beta \underline{p}_i)-jr^{c_i^x}(\beta \overline{p}_i).
\end{equation}
According to \eqref{SAC def}, and substituting Eqs. \eqref{SAC underline p_i}, \eqref{SAC overline p_i}, and \eqref{jr p_i}:
\[
\Delta_{c_{i}^{x}}^{TC_{\emph{Variables}}}\left( \beta \overline{p}_{i},S\right) -\Delta_{c_{i}^{x}}^{TC_{\emph{Variables}}}\left(  \beta \underline{p}_{i},S\right)> \Delta_{c_{i}^{x}}^{TC_{\emph{Variables}}}\left( \overline{p}_{i},S\right) -\Delta_{c_{i}^{x}}^{TC_{\emph{Variables}}}\left( \underline{p}_{i},S\right)>0,
\]
where the second inequality follows from \cite{CY2018}. This completes the proof of Lemma \ref{Lemma Delta p_i}.
\end{proof}
Let us denote the following four solutions
\begin{equation*}%
\begin{array}
[c]{cc}%
\underline{S} & - \ \forall c_i^x \in \emph{Variables}:t_{c^x_i}=\underline{p}_i\\
\overline{S} & - \ \forall c_i^x \in \emph{Variables}:t_{c^x_i}=\overline{p}_i\\
\beta \underline{S} & - \ \forall c_i^x \in \emph{Variables}:t_{c^x_i}=\beta \underline{p}_i\\
\beta \overline{S} & - \ \forall c_i^x \in \emph{Variables}:t_{c^x_i}= \beta \overline{p}_i.
\end{array}
\label{EOQ rounding rules}%
\end{equation*}

Using Lemma \ref{Lemma Delta p_i} we get that in a continuous environment:
\[
LB \left (TC_{\emph{Variables}}\left(  S\right) \right)= TC_{\emph{Variables}}\left(\beta  \underline{S}\right)
\] 
and 
\[
UB \left (TC_{\emph{Variables}}\left(  S\right) \right)= TC_{\emph{Variables}}\left( \beta \overline{S}\right).
\] 
We consider the time series induced by the entire set of \emph{Variables} in a given solution $S$ and denote it $F(S)$.
Let us extend the notation of $jr^{c_i^x}(t_{c_i^x})$ to this set:
\[
jr^{S}(\beta)= UJR\left( F(S)\right)-IJR\left( F(S),\left(\emph{Constants},\beta\right)\right).
\]
In addition, we denote $jr^{c_i^z}(S,\beta)$ to be the marginal addition of a single commodity $c_i^z$ to the joint replenishment cost assuming solution $S$.
\[
jr^{c_i^z}(S,\beta)= UJR\left( F_{t^*_{c_i^z}}\right)-IJR\left( F_{t^*_{c_i^z}},\left(\emph{Constants},\beta\right) \right).
\]

\cite{CY2018} proved that the lower bound on a solution with even one unsynchronized commodity of type \emph{Clauses} is greater than the upper bound on a solution where all commodities of type \emph{Clauses} are satisfied.
That is, given a solution $S$ with a commodity $C_i^z \in Clauses$ that is unsynchronized with any of its associated variable commodities:
\begin{equation*}
LB \left (TC_{\emph{Variables}}\left(  S\right) \right)+LB\left(jr^{c_i^z}(S,1)\right)>UB \left (TC_{\emph{Variables}}\left(  S\right) \right)\\
\end{equation*}
And according to Lemma \ref{Lemma Delta p_i},
 \begin{equation*}
TC_{\emph{Variables}}\left( \underline{S}\right)+LB\left(jr^{c_i^z}(S,1)\right)>TC_{\emph{Variables}}\left(\overline{S}\right).\\
\label{Eq TC discrete final}
\end{equation*}

Based on this, we shall prove the following Lemma:
\begin{lemma}
$TC_{\emph{Variables}}\left( \beta\underline{S}\right)+LB\left(jr^{c_i^z}(S,\beta)\right)-TC_{\emph{Variables}}\left(\beta \overline{S}\right)>0$
\end{lemma}
\begin{proof}
The cost $TC_{\emph{Variables}}\left (\beta S \right)$ is constructed of two elements, the standalone cost and the marginal addition to the number of joint replenishments per period:
\begin{align*}
TC_{\emph{Variables}}\left (\beta S \right)&=\sum \limits_{{c_i^x} \in \emph{Variables} }SAC^{c_i^x}(t_{c_i^x})+K_0 \cdot jr^S(\beta)\\
&=\sum_{c_{i}^{x}\in\emph{Variables}}\left(  \frac{K_{c_{i}^{x}}}%
{t_{c_i^x}}+t_{c_i^x}\cdot h_{c_{i}^{x}}\right)+K_0 \cdot jr^S(\beta).
\end{align*}
Accordingly,
\begin{align*}
&TC_{\emph{Variables}}\left( \beta\underline{S}\right)-TC_{\emph{Variables}}\left(\beta \overline{S}\right)+LB\left(jr^{c_i^z}(S,\beta)\right)\\
&  =\sum_{c_{i}^{x}\in\emph{Variables}}\left(  \frac{K_{c_{i}^{x}}}%
{\beta \underline{p}_{i}}+\beta \underline{p}_{i}\cdot h_{c_{i}^{x}}- \frac{K_{c_{i}^{x}}}%
{\beta \overline{p}_{i}}-\beta\overline{p}_{i}\cdot h_{c_{i}^{x}}\right)
+K_0\left(jr^{\beta \underline{S}}(\beta)-jr^{\beta \overline{S}}(\beta) +LB\left(jr^{c_i^z}(S,\beta)\right)\right).
&
\end{align*}

Substituting for $\overline{p}_{i}=\left( \underline{p}_{i}+b_{i}\right) $
\begin{align}
&TC_{\emph{Variables}}\left( \beta\underline{S}\right)-TC_{\emph{Variables}}\left(\beta \overline{S}\right) +LB\left(jr^{c_i^z}(S,\beta)\right)\nonumber\\
&  =\sum_{c_{i}^{x}\in\emph{Variables}}\left( \frac{b_{i}K_{c_{i}^{x}}%
}{\beta \underline{p}_{i}\left( \underline{p}_{i}+b_{i}\right) }-\beta \cdot b_{i}\cdot
h_{c_{i}^{x}}\right) +K_0\left(jr^{\beta \underline{S}}(\beta)-jr^{\beta \overline{S}}(\beta) +LB\left(jr^{c_i^z}(S,\beta)\right)\right).
\label{LB(TC_Variables(S))-UB(TC_Variables)}%
\end{align}

Substituting for $K_{c_{i}^{x}}$ and $K_{0}$ using Eqs. \eqref{DEF vars K} and \eqref{DEF K_0} into Eq. \eqref{LB(TC_Variables(S))-UB(TC_Variables)} we get:
\begin{align*}
&TC_{\emph{Variables}}\left( \beta\underline{S}\right)-TC_{\emph{Variables}}\left(\beta \overline{S}\right) +LB\left(jr^{c_i^z}(S,\beta)\right)\nonumber\\
&  =\sum_{c_{i}^{x}\in\emph{Variables}}\left(  
b_{i}\cdot
h_{c_{i}^{x}}\left ( \frac{1}{\beta}-\beta\right) 
-\frac{\alpha_{c}%
\overline{\alpha}_{v}b_{i}}{\beta\underline{p}_{i}\left(  \underline{p}_{i}%
+b_{i}-1\right)  }\right)  +\left(jr^{\beta \underline{S}}(\beta)-jr^{\beta \overline{S}}(\beta)+LB\left(jr^{c_i^z}(S,\beta)\right) \right) \nonumber \\
&> \sum_{c_{i}^{x}\in\emph{Variables}}\left(  
\left ( \frac{1}{\beta}-\beta\right) 
-\frac{\alpha_{c}%
\overline{\alpha}_{v}b_{i}}{\beta\underline{p}_{i}\left(  \underline{p}_{i}%
+b_{i}-1\right)  }\right)  +\left(jr^{\beta \underline{S}}(\beta)-jr^{\beta \overline{S}}(\beta)+LB\left(jr^{c_i^z}(S,\beta)\right) \right), \nonumber
\end{align*}
where the inequality comes from $b_{i}\cdot h_{c_{i}^{x}} < 1$ and $\beta \ge 1$ (which makes $\left ( \frac{1}{\beta}-\beta\right) \le 0$).

The joint replenishment frequency scales linearly with the $\beta$, hence the cost is linear to $\frac{1}{\beta}$. Thus:
\begin{align}
&TC_{\emph{Variables}}\left( \beta\underline{S}\right)-TC_{\emph{Variables}}\left(\beta \overline{S}\right) +LB\left(jr^{c_i^z}(S,\beta)\right)  \nonumber\\
&\ge \frac{1}{\beta}\sum_{c_{i}^{x}\in\emph{Variables}}\left( \left ( 1-\beta^2 \right )
-\frac{\alpha_{c}%
\overline{\alpha}_{v}b_{i}}{\underline{p}_{i}\left(  \underline{p}_{i}%
+b_{i}-1\right) }\right) +
\frac{1}{\beta}
\left(jr^{\underline{S}}(1)-jr^{\overline{S}}(1)+LB\left(jr^{c_i^z}(S,1)\right) \right) \nonumber\\
&\ge \frac{1}{\beta}\sum_{c_{i}^{x}\in\emph{Variables}}\left( \left ( 1-\beta
\right )\left ( 1+\beta \right ) -\frac{\alpha_{c}%
\overline{\alpha}_{v}b_{i}}{\underline{p}_{i}\left(  \underline{p}_{i}%
+b_{i}-1\right)  }\right) +
\frac{1}{\beta}
\left(jr^{\underline{S}}(1)-jr^{\overline{S}}(1)+LB\left(jr^{c_i^z}(S,1)\right) \right) \nonumber\\
&> \frac{1}{\beta}\sum_{c_{i}^{x}\in\emph{Variables}}\left( -3\delta
-\frac{\alpha_{c}%
\overline{\alpha}_{v}b_{i}}{\underline{p}_{i}\left( \underline{p}_{i}%
+b_{i}-1\right) }\right) +
\frac{1}{\beta}
\left(jr^{\underline{S}}(1)-jr^{\overline{S}}(1)+LB\left(jr^{c_i^z}(S,1)\right) \right) \nonumber\\
&=\frac{-3n \cdot \delta}{\beta}+ \frac{1}{\beta}\left [ \sum_{c_{i}^{x}\in\emph{Variables}}\left( -\frac{\alpha_{c}%
\overline{\alpha}_{v}b_{i}}{\underline{p}_{i}\left(  \underline{p}_{i}%
+b_{i}-1\right)  }\right)  +
\left(jr^{\underline{S}}(1)-jr^{\overline{S}}(1)+LB\left(jr^{c_i^z}(S,1)\right) \right) \right ], \label{Eq: difference}
\end{align}
where the last inequality comes from $(1-\beta >-\delta)$ and $\beta<2$.
The term in square brackets in Eq. \eqref{Eq: difference} is, in fact, the term $TC_{\emph{Variables}}\left( \underline{S}\right)-TC_{\emph{Variables}}\left( \overline{S}\right) +LB\left(jr^{c_i^z}(S,1)\right)$. \cite{CY2018} showed that:
\begin{equation}
TC_{\emph{Variables}}\left( \underline{S}\right)-TC_{\emph{Variables}}\left( \overline{S}\right) +LB\left(jr^{c_i^z}(S,1)\right) \ge \frac{1}{\left (\overline{p}_n \right)^6}.
\label{Eq: diff without beta}
\end{equation}
Substituting for $\delta$ using Eq. \eqref{eq: delta} and using Eq. \eqref{Eq: diff without beta} to replace the square brackets in Eq. \eqref{Eq: difference} we get:
\begin{align*}
&TC_{\emph{Variables}}\left( \beta\underline{S}\right)-TC_{\emph{Variables}}\left(\beta \overline{S}\right) +LB\left(jr^{c_i^z}(S,\beta)\right)  \nonumber\\ 
& > \frac{1}{\beta} \left (\frac{-3n}{6n\left (\overline{p}_n \right)^6}+\frac{1}{\left (\overline{p}_n \right)^6}\right )\\
&>\frac{1}{2} \cdot \frac{1}{2\left (\overline{p}_n \right)^6}= \frac{1}{4\left (\overline{p}_n \right)^6}>0,
\end{align*}
where the second inequality is true since $\beta<2$. This completes our proof.

\end{proof}

\bibliographystyle{informs2014} 
\bibliography{general0.bib}

\end{document}